\documentclass{article}
\usepackage[utf8]{inputenc}
\usepackage{amsmath,amsthm,amsfonts,amssymb,amstext}

\title{The union-closed sets conjecture almost holds for almost all random bipartite graphs}
\author{Henning Bruhn and Oliver Schaudt}
\date{}

\begin{document}

\newtheorem{theorem}{Theorem}
\newtheorem{lemma}[theorem]{Lemma}
\newtheorem{observation}[theorem]{Observation}
\newtheorem{corollary}[theorem]{Corollary}
\newtheorem{claim}[theorem]{Claim}
\newtheorem{conjecture}[theorem]{Conjecture}

\def\exp{\mbox{\rm{E}}}
\def\pr{\mbox{\rm Pr}}

\def\stab{\mbox{stab}}

\newcommand{\smallmss}{\ensuremath{\mathcal{S}_G}}

\newcommand{\mymargin}[1]{%
  \marginpar{%
    \begin{minipage}{\marginparwidth}\small%
      \begin{flushleft}%
        #1%
      \end{flushleft}%
   \end{minipage}%
  }%
}%

\renewcommand{\epsilon}{\varepsilon}
\newcommand{\rbip}{\mathcal B}

\newcommand{\mstab}{\mathcal{A}}
\newcommand{\lrside}{{\lfloor\log_{1/q}(m)\rfloor}}
\newcommand{\lmss}{\mathcal L_G}
\newcommand{\lmssthird}{\lmss'}
\newcommand{\lmsshalf}{\lmss}
\newcommand{\lmssdelta}{\lmss^\delta}

\newcommand{\sm}{\setminus}
\newcommand{\av}{\mbox{\rm{left-avg}\ensuremath{(G)}}}
\newcommand{\rav}{\mbox{\rm{right-avg}(G)}}
\newcommand{\hoeffmss}{\ensuremath{\mathcal{S}'_G}}
\newcommand{\Lside}{\lfloor\log_{1/q}(n)\rfloor}
\newcommand{\Rside}{\lfloor\log_{1/q}(m)\rfloor}

\newcommand{\comment}[1]{{\newline\bigskip 
\marginpar{\bf COMMENT} \em #1\newline\bigskip}}
\renewcommand{\comment}[1]{}

\maketitle

\begin{abstract}
Frankl's union-closed sets conjecture states that in every
 finite union-closed set of sets, there is an element that is contained 
in at least half of the member-sets 
(provided there are at least two members).
The conjecture has an equivalent formulation in terms of graphs: 
In every bipartite graph with least one edge, both colour classes contain a vertex belonging to at most half of the maximal stable sets. 

We prove that, for every fixed edge-probability,
almost every random bipartite graph almost satisfies Frankl's conjecture.



\end{abstract}

\section{Introduction}

One of the most basic conjectures in 
extremal set theory is Frankl's conjecture on union-closed set 
systems. 
A set $\mathcal X$ of sets is  \textit{union-closed} if 
$X \cup Y \in \mathcal X$
for all $X,Y \in\mathcal X$.

\newtheorem*{ucsc}{Union-closed sets conjecture}
\begin{ucsc}
Let $\mathcal X\neq \{ \emptyset\}$ be a finite union-closed set of  sets. 
Then there is a $x\in\bigcup_{X\in\mathcal X}X$
that lies in at least half of the members of~$\mathcal X$.
\end{ucsc}

While Frankl~\cite{Frankl} dates the conjecture to 1979, 
it apparently did not appear in print before
1985, when it was mentioned as an open problem in Rival~\cite{Riv85}. 
Despite being widely known, there is only little substantial progress on the conjecture.
 
The conjecture has two equivalent formulations, one in terms of lattices and one 
in terms of graphs. For the latter, let us say that 
a  vertex set $S$ in a graph is \emph{stable} if 
no two of its vertices are adjacent, and that it is \emph{maximally stable} if, in 
addition, every vertex outside $S$ has a neighbour in~$S$.

\begin{conjecture}[Bruhn, Charbit and Telle~\cite{BCT12}]\label{mstabconj}
Let $G$ be a bipartite graph with at least one edge. 
Then each of the two bipartition classes contains a vertex 
belonging to at most half of the maximal stable sets. 
\end{conjecture}

We prove a slight weakening of Conjecture~\ref{mstabconj} for random bipartite graphs.
For $\delta>0$,
we say that a bipartite graph \emph{satisfies the union-closed sets conjecture up to~$\delta$}
if 
each of its two bipartition classes has a vertex
for which the number of maximal stable sets containing it is at most
$\tfrac{1}{2}+\delta$ times the total number 
of maximal stable sets. 
A \emph{random bipartite graph} is a graph on bipartition classes of cardinalities $m$ and $n$, 
where any two vertices from different classes are independently joined by an 
edge with probability $p$.
We say that \textit{almost every} random bipartite graph has property $P$ if 
for every $\epsilon >0$ there is an $N$ such that, whenever $m+n \ge N$, 
the probability that a random bipartite graph on $m+n$ vertices 
has $P$ 
is at least $1-\epsilon$.

We prove:
\begin{theorem}\label{intromainthm}
Let $p\in (0,1)$ be a  fixed edge-probability.
For every $\delta>0$, almost every random bipartite graph satisfies the union-closed sets conjecture up to~$\delta$.
\end{theorem}

While Frankl's conjecture has attracted quite a lot of interest,  
a proof seems still  out of reach.
For a fairly complete bibliography on the  conjecture, see 
Cz\'edli, Mar\'oti and Schmidt~\cite{CMS09}.

Some of the earliest results verified the conjecture for few sets or few 
elements in the ground set, that is, when $n=|\mathcal X|$ or $m=|\bigcup_{X\in\mathcal X}X|$ are small. 
The current best results show that the conjecture holds for $m\leq 11$, 
which is due to Bo\v{s}njak and Markovi\'{c}~\cite{BM08}, and for $n\leq 46$, proved by Nishimura and Takahashi~\cite{NT96} and independently Roberts and Simpson~\cite{RSi10}. 
The conjecture is also known to be true when $n$ is large compared to $n$, 
that is $n\geq 2^m-\tfrac{1}{2}\sqrt{2^m}$ (Nishimura and Takahashi~\cite{NT96}).
The latter result was improved upon by Cz\'edli~\cite{Cze09}, who shows 
that $n\geq 2^m-\sqrt{2^m}$
is enough. 
Recently, Balla, Bollobas and Eccles~\cite{BBE13} pushed this to $n \ge \lceil \tfrac{1}{3}2^{m+1} \rceil$.

The lattice formulation of the conjecture was apparently known from very early on, 
as it is already mentioned in Rival~\cite{Riv85}. 
Poonen~\cite{Poo92} investigated several variants and gave proofs for geometric as 
well as distributive lattices. Reinhold~\cite{Rei00} extended this, with a very concise argument, to 
lower semimodular lattices. Finally, the conjecture holds as well for 
large semimodular lattices and for planar semimodular lattices (Cz\'edli and Schmidt~\cite{CS08}).

The third view, in terms of graphs, on the union-closed sets conjecture
is more recent. So far, the graph formulation is only verified for 
chordal-bipartite graphs and for bipartitioned circular interval graphs 
(Bruhn, Charbit and Telle~\cite{BCT12}).

One of the main techniques that is used for the set formulation of 
Frankl's conjecture as well as for 
the lattice formulation, is averaging: The average  frequency of an element 
is computed, and if that average is at least half of the size of the set
system, it is concluded that the conjecture holds for the set system.
Averaging is also our main tool. We discuss averaging and its limits in Section~\ref{avsec}.

\section{Basic tools and definitions}

In our graph-theoretic notation we usually follow Diestel~\cite{diestelBook10}, while we 
refer to Bollob\'as~\cite{BolloBook01} for more details on random graphs.

All our graphs are finite and simple. We always consider a bipartite graph $G$
to have a fixed bipartition, which we denote by $(L(G),R(G))$. When discussing the 
bipartition classes, we will often refer to $L(G)$ as the \emph{left side}
and to $R(G)$ as the \emph{right side} of the graph.

Throughout the paper we consider a fixed edge probability $p$ with $0 < p < 1$;
and we will always put $q = 1-p$.
A \textit{random bipartite graph} $G$ 
is a bipartite graph where 
every pair $u \in L(G)$ and $v \in R(G)$ is joined by an edge independently with probability $p$.
We denote by $\rbip(m,n;p)$ 
the probability space whose elements are the random bipartite
graphs $G$ with $|L(G)|=m$ and $|R(G)|=n$.
We will always tacitly assume that $m\geq 1$ and $n\geq 1$. Indeed, 
if one of the sides of the random bipartite graph is empty, then the graph has no 
edge and is therefore trivial with respect to Conjecture~\ref{mstabconj}.

Markov's inequality states that for a non-negative random variable $X$ and any $\alpha > 0$, 
\begin{equation}\label{markov}
\pr[X \ge \alpha] \le \frac{\exp[X]}{\alpha}.
\end{equation}

Chebyshev's inequality is as follows.
Let $X$ be a random variable with finite variance $\sigma^2 = \exp[X^2] - \exp[X]^2$. 
Then, for every real $\lambda > 0$,
\begin{equation}\label{cheb}
\pr[|X-\exp[X]|\geq \lambda] \leq \frac{\sigma^2}{\lambda^2}. 
\end{equation}

%

\section{Discussion of averaging}
\label{avsec}

Most of the partial results on Frankl's conjecture are based on one of
two techniques: \emph{Local configurations} and \emph{averaging}. By a local configuration 
we mean a subsystem of the union-closed set system $\mathcal X$, that guarantees
that one element of the ground set lies in at least half of the members of $\mathcal X$. 
For instance, one of the earliest results is the observation of Sarvate and Renaud~\cite{SR90b}
that the element of a singleton will always
belong to at least half of the sets.  
More local configurations have later been found by Poonen~\cite{Poo92}, 
Vaughan~\cite{Vaug02}, Morris~\cite{Mor04} and others.

The second technique consists in taking the average of the number of member sets
containing a given element, where the average ranges over the set 
$U=\bigcup_{X\in\mathcal X}X$ of all elements. If that average is at least $\tfrac{1}{2}|U|$ 
then clearly $\mathcal X$ will satisfy the conjecture. 
Averaging was used successfully by  Cz\'edli~\cite{Cze09} to show that the conjecture
holds when there are vastly more sets than elements, 
that is, when $|\mathcal X|\geq 2^{|U|}-\sqrt{2^{|U|}}$.
Reimer~\cite{Rei03} showed that the average is always at least $\log_2(|U|)$. 

Averaging will not always work. It is easy to construct union-closed set systems
in which the average is too low.
Cz{\'e}dli, Mar{\'o}ti and Schmidt~\cite{CMS09} 
even found such  set systems of size 
$|\mathcal X|=\lfloor 2^{|U|+1}/3\rfloor$.
Nevertheless, we will see that, in the graph formulation, 
averaging will almost always allow us to conclude that 
the union-closed sets conjecture is satisfied (up to any $\delta>0$).
\medskip

To describe the averaging technique for bipartite graphs, let us 
write $\mstab(G)$ 
for the set of maximal stable sets of a bipartite graph $G$. 
The graph formulation of the union-closed sets conjecture, Conjecture~\ref{mstabconj},
is satisfied if $G$ contains an \emph{unstable} vertex in both bipartition classes, 
that is, a vertex that lies in at most half of the maximal stable sets.
We note first that exchanging the sides turns a random bipartite graph $G\in\rbip(m,n;p)$
into a member of $\rbip(n,m;p)$, which means that it will suffice to show the 
existence of an unstable vertex in $L(G)$. All the discussion and proofs that 
follows will focus on the left side $L(G)$. 

That a vertex $v$ is unstable means that $|\mstab_v(G)|$, the number of maximal stable 
sets containing $v$, is at most $\tfrac{1}{2}|\mstab(G)|$. Thus, if for the average
\[
\sum_{v\in L(G)}\frac{|\mstab_v(G)|}{|\mstab(G)|}\leq \frac{1}{2}|L(G)|
\]
then $L(G)$ will contain an unstable vertex. Double-counting shows that 
the above average is equal to 
\[
\av:= \sum_{A\in\mstab(G)}\frac{|A\cap L(G)|}{|\mstab(G)|},
\]
and thus our aim is to show that when $m+n$ is very large,
it follows with  high probability that
$\av\leq\tfrac{m}{2}$  for any
$G\in\rbip(m,n;p)$. 

Unfortunately, we will not reach this aim. While we will show for
large parts of the parameter space $(m,n)$ that the 
average is, with high probability, small enough, we will also see
that when $n$ is roughly $q^{-\frac{m}{2}}$ the average becomes very close 
to $\tfrac{m}{2}$, so close that our tools are not sharp enough to separate
the average from slightly above $\tfrac{m}{2}$. 
Therefore, we provide for a bit more space by settling on bounding 
the average away from $(\tfrac{1}{2}+\delta)m$ for any positive $\delta$,
which then only allows us to deduce the existence of a vertex $v\in L(G)$
that is \emph{almost unstable}, in the sense that $v$ lies in at most 
$(\tfrac{1}{2}+\delta)|\mstab(G)|$ maximal stable sets. 

Much of the previous discussion is subsumed in the following lemma.

\begin{lemma}\label{avlem}
Let $G$ be a bipartite graph, and let $\delta\geq 0$.
If 
\[
\av\leq 
\left(\tfrac{1}{2}+\delta\right)|L(G)|
\]
then there exists a vertex in $L(G)$ that lies in at most 
$\left(\tfrac{1}{2}+\delta\right)|\mstab(G)|$ maximal stable sets.
\end{lemma}
\begin{proof}
Double counting yields 
$\sum_{y\in L(G)}|\mstab_y(G)|=\sum_{A\in\mstab(G)}|A\cap L(G)|$, from
which we deduce that 
\(
\sum_{y\in L(G)}|\mstab_y(G)|\leq |L(G)|\cdot
\left(\tfrac{1}{2}+\delta\right)|\mstab(G)|.
\)
Thus there is a $y\in L(G)$ with $|\mstab_y(G)|\leq 
\left(\tfrac{1}{2}+\delta\right)|\mstab(G)|$.
\end{proof}

Most of the effort in this article will be spent on proving the 
following result, which is the heart of our main result, Theorem~\ref{intromainthm}:

\begin{theorem}\label{realmainthm}
For all $\delta>0$ and all $\epsilon>0$ there is an integer $N$ so 
that for $G\in \rbip(m,n;p)$ 
\[
\pr\left[\av\leq \left(\tfrac{1}{2}+\delta\right)m\right]\geq 1-\epsilon
\]
for all $m,n$ with $m+n\geq N$ and $n\geq \max\{20, (\lceil 3\log_{1/q}(2)\rceil+2)^2\}+1$.
\end{theorem}

In order to show how Theorem~\ref{intromainthm} 
follows from Theorem~\ref{realmainthm}, we need to deal with the special case
when one side is of constant size while the other becomes ever larger. 
Indeed, in this case averaging might fail---for a trivial reason. 
If we fix a constant right side $R(G)$, while $L(G)$ becomes ever larger, then $L(G)$ will 
contain many isolated vertices. 
Since the isolated vertices lie in every maximal stable set
they may push up $\av$ to above $\tfrac{m}{2}$. 

However, isolated vertices are never a threat to 
Frankl's conjecture: A bipartite graph satisfies the union-closed
sets conjecture if and only if it satisfies the conjecture 
with all isolated vertices deleted. 
More generally, it turns out that the special case of a constant
right side is easily taken care of:

\begin{lemma}\label{constrightside}
Let $c$ be a positive integer, and let $\epsilon>0$. Then 
there is an $N$ so that for $G\in\rbip(m,n;p)$
\[
\pr\left[L(G)\text{\rm\ contains an unstable vertex}\right]\geq 1-\epsilon,
\]
for all $m,n$ with $m\geq N$ and $n\leq c$.
\end{lemma}

\begin{proof}
Let $G$ be any bipartite graph, and
suppose there is a vertex $v\in L(G)$ that is adjacent with every
vertex in $R(G)$. Then, the only maximal stable set that contains 
$v$ is $L(G)$. Since the fact that $v$ is incident with an edge
implies that $G$ has at least two maximal stable sets, $v$ is unstable.

We now calculate the probability that there is such a vertex. 
The probability that $R(G)=N(v)$ for a fixed vertex $v\in L(G)$ is $p^{n}\geq p^c$ if $n\leq c$.
Thus the probability that no such vertex exists in $L(G)$ is at most $(1-p^c)^m$, which tends to~$0$ as $m\to\infty$. 
\end{proof}

\begin{proof}[Proof of Theorem~\ref{intromainthm}]
Let $\delta>0$ be given.
By symmetry, it is enough to show that the left side $L(G)$
of almost every random bipartite graph $G$ in $\rbip(m,n;p)$ 
contains  a vertex that lies in at most $(\tfrac{1}{2}+\delta)|\mathcal A|$
maximal stable sets. For this, consider a $\epsilon>0$, and let 
$N$ be the maximum of the $N$ given by Theorem~\ref{realmainthm}
and Lemma~\ref{constrightside} with $c=\max\{20, (\lceil 3\log_{1/q}(2)\rceil+2)^2\}$.
Consider a pair $m,n$ of positive integers with $m+n\geq N$.
If $n\leq\max\{20, (\lceil 3\log_{1/q}(2)\rceil+2)^2\}$ then Lemma~\ref{constrightside}
yields an unstable vertex in $L(G)$ with probability at least $1-\epsilon$.
If, on the other hand, $n\geq \max\{20, (\lceil 3\log_{1/q}(2)\rceil+2)^2\}+1$,
Theorem~\ref{realmainthm} becomes applicable, which is to say 
that with probability at least $1-\epsilon$ 
we have $\av\leq\left(\tfrac{1}{2}+\delta\right)m$. 
Now, Lemma~\ref{avlem} yields the desired vertex in $L(G)$.
\end{proof}

\medskip
We close this section with the obvious but useful 
observation that if there are many more maximal stable sets
with small left side than with large left side, then 
the average over the left sides is small, too.
We will use this lemma repeatedly.

\begin{lemma}\label{largeandsmall}
Let $\nu>0$ and $\delta\geq 0$, and
let $G$ be a bipartite graph with $|L(G)|=m$.
Let $\mathcal L$ be the maximal stable sets $A$ of $G$
with $|A\cap L(G)|\geq (\tfrac{1}{2}+\delta)m$, and let $\mathcal S$ 
be those maximal stable sets $B$ with $|B\cap L(G)|\leq (1-\nu)\tfrac{m}{2}$. 
If $|\mathcal S|\geq \tfrac{1}{\nu}|\mathcal L|$ then
\[
\av\leq \left(\tfrac{1}{2}+\delta\right)m.
\]
\end{lemma}
\begin{proof}
Let $\mathcal M=\mstab(G)\sm (\mathcal L\cup\mathcal S)$, that is, 
$\mathcal M$ is the set of those maximal stable sets $A$ 
with $(1-\nu)\tfrac{m}{2}< |A\cap L(G)|<(\tfrac{1}{2}+\delta)m$.
Then 
\begin{align*}
\sum_{A\in\mstab(G)}\frac{|A\cap L(G)|}{\left(\frac{1}{2}+\delta\right)m}
& \leq \sum_{A\in\mathcal L} \frac{m}{\left(\frac{1}{2}+\delta\right)m}+
\sum_{B\in\mathcal M} \frac{(\frac{1}{2}+\delta)m}{\left(\frac{1}{2}+\delta\right)m}+
\sum_{C\in\mathcal S} \frac{(1-\nu)\frac{m}{2}}{\left(\frac{1}{2}+\delta\right)m}\\
& \leq 2|\mathcal L|+|\mathcal M|+(1-\nu)|\mathcal S| 
= |\mstab(G)| - (\nu|\mathcal S|-|\mathcal L|).
\end{align*}
Thus, from $|\mathcal S|\geq \tfrac{1}{\nu}|\mathcal L|$
it follows that $\sum_{A\in\mstab(G)}\frac{|A\cap L(G)|}{\left(\frac{1}{2}+\delta\right)m}\leq |\mstab(G)|$, 
which is equivalent to the inequality of the lemma.
\end{proof}

\section{Proof of Theorem~\ref{realmainthm}}

In order to prove Theorem~\ref{realmainthm}, we distinguish several cases, depending on the 
relative sizes, $m$ and $n$, of the two sides of the random bipartite graph $G\in\rbip(m,n;p)$.
In each of the cases we need a different method.

The general strategy follows Lemma~\ref{largeandsmall}: We bound the number 
of maximal stable sets with large left side, usually counted by a random variable $\lmss$, 
and at the same time we show that there are many maximal stable sets with a small left side;
those we count with $\smallmss$. 

Up to $n<q^{-\frac{m}{2}}$ we are able to use the same bound for 
the number $\lmss$ of maximal stable sets whose left sides are of size at least $\tfrac{m}{2}$:
We prove that with high probability $\lmss$ is bounded by a polynomial in $n$. 
For right sides that are much larger than the left side, i.e.\ $m\gg n$,
we even extend such a bound to 
maximal stable sets with left side $\geq\tfrac{m}{3}$.

For the maximal stable sets with small left side, counted by $\smallmss$, we need to distinguish 
several cases. When the left side of the graph is much larger than the right side, 
namely $m \ge q^{-\sqrt[5]{n}}$, we find with high probability a large induced matching in $G$. 
This in turn implies that the total number of maximal stable sets is high, and thus
clearly also the number of those with small left side.

When the sides of the graph do not differ too much in size, 
$m \le q^{-\sqrt[5]{n}}$ and $n \le q^{-\sqrt[5]{m}}$,
the variance of the number of maximal stable sets with small left side is moderate enough 
to apply Chebychev's inequality. Since the expectation of $\smallmss$ 
is high, we again can use Lemma~\ref{largeandsmall} to deduce Theorem~\ref{realmainthm}.

However, when the left side of the graph becomes much larger than the right side, 
we cannot control the variance of $\smallmss$ anymore. Instead, 
for $q^{-\sqrt[5]{m}} \leq n \leq q^{-\frac{m}{16}}$,
we cut the right side into many pieces each of large size and apply Hoeffding's inequality
to each of the pieces together with the left side. The inequality ensures that 
we find on at least one of the pieces a large number of maximal stable sets 
of small left side.
Surpassing $n\geq q^{-\frac{m}{16}}$, we have to refine our estimations 
but we can still use this strategy up to slightly below $n=q^{-\frac{m}{2}}$. 

In the interval $q^{-\frac{m}{2}}\leq n\leq q^{-m^3}$, we encounter a serious obstacle.
There,
we have to cope with an average
that is very close to $\tfrac{m}{2}$. 
It is precisely for this reason that, overall, we only prove 
that $\av\leq\left(\tfrac{1}{2}+\delta\right)m$ instead of $\av\leq\tfrac{m}{2}$.
To keep below the slightly higher average, we only need to bound 
the number of maximal stable sets with left side $> (\tfrac{1}{2}+\delta)m$. 
This number we will almost trivally bound by $2^{\lambda m}$, with some $\lambda<1$. 
On the other hand,
we will see that 
the number $\smallmss$ of maximal stable sets of small left side is $2^{\lambda' m}$
with a $\lambda'$ as close to~$1$ as we want.

In the remaining case, we are dealing with an enormous right side: $n\geq q^{-m^3}$. 
Then, it is easy to see that with high probability there is an induced matching 
that covers all of the left side, which implies that every subset of $L(G)$
is the left side of a maximal stable set. 
This immediately gives us $\av =\tfrac{m}{2}$.

\subsection{The case $m \ge q^{-\sqrt[5]{n}}$}\label{sect.largeleft}

In this section we treat the graphs whose left side 
is much larger than the right side. From an easy argument it 
follows that, with high probability, any large enough random graph $G\in \rbip(m,n;p)$
contains an induced matching of size $s\log_2(n)$, for any constant $s$. 
This directly implies that the total number  of maximal stable sets
is large. 
At the same time, we shall bound the number of maximal stable sets of large 
left side, which then shows that there are many of small left side. 

However, if we take \emph{large} left side to mean at least $\tfrac{m}{2}$ then it might
be that most of those of \emph{small} left side have a left side whose 
size is very close to $\tfrac{m}{2}$. Such a left side does not help 
much to drop the average. Therefore, we will consider a more generous 
notion of a large left side and bound the number of maximal 
stable sets that have a left side of $\geq \tfrac{m}{3}$; then small will 
mean $<\tfrac{m}{3}$. 


\medskip
For a random graph $G\in\rbip(m,n;p)$,
let $\stab(\geq \ell;\geq r)$
denote the number of stable sets of  
that have at least $\ell$ vertices in $L(G)$ and at least $r$ vertices in~$R(G)$.

\begin{lemma}\label{genupper}
Let  $\ell^*\leq m$ and $r^*\leq n$ 
so that 
\(
nq^{\ell^*} \leq \frac{1}{2}.
\) 
Then for $G\in\rbip(m,n;p)$
\[
\exp[\stab(\geq \ell^*;\geq r^*)]\leq 
2^{m+1} \left(nq^{\ell^*}\right)^{r^*}.
\]
\end{lemma}
\begin{proof}
The expectation is given by
\begin{align*}
\exp[\stab(\geq \ell^*;\geq r^*)] &=
\sum_{\ell=\ell^*}^m\sum_{r=r^*}^n {m\choose \ell}{n\choose r} q^{\ell r}\\
&\leq
\left(\sum_{\ell=\ell^*}^m{m\choose \ell}\right)
\left(\sum_{r=r^*}^n {n\choose r} q^{\ell^* r}\right)\\
&\leq 2^m \left(
\sum_{r=r^*}^{n}  \left(nq^{\ell^*}\right)^r
\right) \leq 
2^m \left(
\sum_{r=r^*}^{\infty}  \left(nq^{\ell^*}\right)^r
\right).
\end{align*}

The error estimation for the geometric series
yields $\sum_{i=k}^\infty z^i\leq 2|z|^k$
for any $|z|\leq\tfrac{1}{2}$. Applying this for $z=nq^{\ell^*}$, 
we obtain the claimed bound of the lemma.
\end{proof}

In the following, we denote by $\lmssthird$ 
the number of 
maximal stable sets $S$ of a bipartite graph $G$ with $|S \cap L(G)| \geq \tfrac{m}{3}$.

\begin{lemma}\label{largeleftupper}
Let $r^*=\lceil 3\log_{1/q}(2)\rceil+1.$ Then 
for any $\epsilon>0$ there is an $N$ so that for $G\in\rbip(m,n;p)$
\[
\pr[\lmssthird \le n^{r^*}]\ge 1-\epsilon 
\]
for all $m,n$ with $m\geq q^{-\sqrt[5]{n}}$ and $m+n\geq N$. 
\end{lemma}
\begin{proof}
Throughout the proof we assume that $m\geq q^{-\sqrt[5]{n}}$.

Setting $\ell^*=\tfrac{m}{3}$, we get from Lemma~\ref{genupper}
that 
\begin{align*}
\exp[\stab(\geq\tfrac{m}{3};\geq r^*)] &\leq 
2^{m+1}\left(nq^{\frac{m}{3}}\right)^{r^*}\\
&= 2n^{r^*}\cdot q^{m(\frac{r^*}{3}-\log_{1/q}(2))}
\leq 2n^{r^*} \cdot q^{\frac{m}{3}},
\end{align*}
by choice of $r^*$. 

Choose $N$ so that 
$2n^{r^*} \cdot q^{\frac{m}{3}}\leq \epsilon$ for all $m,n$
with $m+n \ge N$. 
Then, from Markov's inequality it follows that 
\begin{equation}\label{mthirds}
\pr[\stab(\geq\tfrac{m}{3};\geq r^*)> 0]\leq \epsilon.
\end{equation}

The set of maximal stable sets $S$ of $G$ whose left side $S\cap L(G)$
has size at least $\tfrac{m}{3}$ is divided into those $S$ with $|S\cap R(G)|\geq r^*$
and those whose right sides have $<r^*$ vertices; let
the number of the latter ones be $t$. 
Since there are at most $n^{r^*}$ subsets of $R(G)$ with at most $r^*$ vertices, $t\le n^{r^*}$.
Hence,  
\begin{align*}
\lmssthird &\leq \stab(\geq\tfrac{m}{3};\geq r^*) + t \\
&\leq \stab(\geq\tfrac{m}{3};\geq r^*) + n^{r^*}.
\end{align*}
From~\eqref{mthirds}, we deduce 
$\pr[\lmssthird>n^{r^*}]\leq\epsilon $.
\end{proof}

\begin{lemma}\label{indmatchings}
Let $s$ be a positive integer, and let $\epsilon>0$. 
Then there is an $N$ so that for $G\in\rbip(m,n;p)$
\[
\pr\left[\text{\rm $G$ has an induced matching of size
$\geq s\log_2(n)$}\right]\geq 1-\epsilon
\]
for all $m,n$ with $m+n\geq N$, $n\geq s\log_2(n)$ and $m\geq q^{-\sqrt[5]{n}}$.
\end{lemma}
\begin{proof}
Let us assume throughout the proof that $n\geq s\log_2(n)$ and $m\geq q^{-\sqrt[5]{n}}$.

Put $k:=\lceil s\log_2(n)\rceil$, and choose
$\lfloor m/k\rfloor$ pairwise disjoint subsets
$L_1,\ldots,L_{\lfloor m/k\rfloor}$
of size $k$ of $L(G)$.
Since $n\geq s\log_2(n)$, $n\geq k$ and so we may choose a set $R' \subseteq R(G)$ with $|R'|=k$.
For $i=1,\ldots,\lfloor m/k\rfloor$ let $M_i$
be the random indicator variable for an induced matching of size $k$ on $L_i\cup R'$.
It is straightforward that
\[
\pr\left[M_i=1\right]=k!p^kq^{k^2-k}\geq p^kq^{k^2}.
\] 
Since the $M_i$ are independent, 
\[
\pr\left[\sum_{i=1}^{\lfloor m/k\rfloor}M_i=0\right]\leq \left(1-p^kq^{k^2} \right)^{\lfloor m/k\rfloor}
\leq e^{-p^kq^{k^2}{\lfloor m/k\rfloor}},
\]
using the standard inequality $1-x \le e^x$ for all $x<1$.
Now for large $m+n$ the dominating term in
\(
p^kq^{k^2}{\lfloor m/k\rfloor}
\)
is $q^{k^2}m$, since $m\geq q^{-\sqrt[5]{n}}$, which becomes arbitrarily large for large $m+n$
as $k = \lceil s\log_2(n)\rceil$ and $n \leq \left(\log_{1/q}(m)\right)^5$. 
Thus, there is an $N$ so that 
$\pr\left[\sum_{i=1}^{\lfloor m/k\rfloor}M_i=0\right]\leq\epsilon$
for all $m,n$ with $m+n\geq N$.
\end{proof}

We have now bounded the number $\lmssthird$ of maximal stable sets 
of large left side,  while the previous lemma will let us 
to conclude that the number of those with small left side is large. 
Together this allows us prove the first case of Theorem~\ref{realmainthm}:

\begin{lemma}\label{largeleftside}
For every $\epsilon>0$ there exists an $N$ so that 
for $G\in\rbip(m,n;p)$ 
\[
\pr\left[\av\leq\tfrac{m}{2}\right]\geq 1-\epsilon,
\]
for all $m,n$ with $m+n\geq N$, $n\geq\max\{20, (\lceil 3\log_{1/q}(2)\rceil+2)^2\}$ 
and $m\geq q^{-\sqrt[5]{n}}$.
\end{lemma}
\begin{proof}
Set $r^*=\lceil 3\log_{1/q}(2)\rceil+1$.
Choose $N$ to be the maximum 
of the $N$ obtained from Lemma~\ref{largeleftupper} for $\tfrac{\epsilon}{2}$
and the one from Lemma~\ref{indmatchings} applied
to $s=r^*+1=\lceil 3\log_{1/q}(2)\rceil+2$ and $\tfrac{\epsilon}{2}$.

Now, consider $m,n$ with $m+n\geq N$, $n\geq\max\{20, (\lceil 3\log_{1/q}(2)\rceil+2)^2\}$ 
and $m\geq q^{-\sqrt[5]{n}}$.
We note that $n\geq\max\{20, (\lceil 3\log_{1/q}(2)\rceil+2)^2\}$
implies that $n\geq s\log_2(n)$.
By choice of $s$ and $N$, we obtain from 
Lemma~\ref{indmatchings} that 
the probability that $G\in\rbip(m,n;p)$
does not contain an induced matching of size at least $s\log_2(n)$ 
is at most $\tfrac{\epsilon}{2}$. On the other hand, the probability 
that the number $\lmssthird$ 
of maximal stable sets $A$ with $|A\cap L(G)|\geq \tfrac{m}{3}$
surpasses $n^{r^*}$ is as well $\leq\tfrac{\epsilon}{2}$. 
Thus, the probability that none of these two events occur
is at least $1-\epsilon$. 
We claim that in this case 
$\av\leq\tfrac{m}{2}$.

So, assume that $G$ contains an induced matching of 
cardinality $\geq s\log_2(n)$ and that $\lmssthird \leq n^{r^*}$.
There are at least $2^{s\log_2(n)}$  maximal stable sets
on the subgraph restricted to the matching edges. Since
each extends to a distinct maximal stable set of $G$, 
the number of maximal stable sets of $G$ is at least
$2^{s\log_2(n)}=n^s=n^{r^*+1}$. 
On the other hand, 
from $\lmssthird \leq n^{r^*}$ it follows 
that at least $n^{r^*}(n-1)$ of the maximal stable sets
have a left side of size at most $\tfrac{m}{3}$. 
As $n-1\geq 3$, we may apply Lemma~\ref{largeandsmall}
with $\delta=0$
in order to see that $\av\leq\tfrac{m}{2}$.
\end{proof}

The key observation in the argument above is that the number of 
maximal stable sets with a large left side is bounded by a polynomial 
in $n$, the size of the right-hand side. We will continue to 
exploit this, in a slightly strengthened version, below. The second part 
of the argument here is to note that there is always a relatively
large induced matching, from which we deduce that the total number
of maximal stable sets is not too small. Then we also have a large 
number of maximal stable sets with small left side, so that we
are guaranteed a small average. 

This strategy fails once $m$ becomes smaller than $n$. 
Assume $m<n$ and, for simplicity,  $p=q=\tfrac{1}{2}$. Below 
we will in that case bound the number of maximal stable sets 
with large left side by about $2n^2$. Thus, for our strategy to 
work, we should better find an induced matching of size 
at least $2\log_2(n)$. An easy calculation, however, shows 
that the expected number 
of induced matchings of size $2\log_2(n)$ is below one.

\subsection{The case $n \le q^{-\sqrt[5]{m}}$ and $m \le q^{-\sqrt[5]{n}}$}\label{sect.5th.sqrt.m}

From now on we will denote by $\lmsshalf$ the number of maximal stable sets $S$
with $|S \cap L(G)| \geq \tfrac{m}{2}$ of a 
random bipartite graph $G\in\rbip(m,n;p)$.
We first bound $\lmsshalf$ by a polynomial in $n$, a bound that will 
be useful up to slightly below $n=2^{\frac{m}{2}}$.

\begin{lemma}\label{squpperbound}
For every $\alpha<\tfrac{1}{2}$ and every $\epsilon>0$ there exists an $N$ so that 
for $G\in\rbip(m,n;p)$
\[
\pr[\lmsshalf \le 2n^{\log_{q}(1/4)}] \geq 1-\epsilon
\]
when $m+n\geq N$ and $n\leq q^{-\alpha m}$.
\end{lemma}

\begin{proof}
Let $\alpha<\tfrac{1}{2}$ be given and assume $n\leq q^{-\alpha m}$.

We determine first the probability that a random bipartite graph contains many 
stable sets (not necessarily maximal) 
with left side $\geq\tfrac{m}{2}$ and right side $\geq \lfloor \log_q(1/4) \rfloor + 1$.

For this, note that  $\alpha < \tfrac{1}{2}$ implies 
 $nq^{\frac{m}{2}}\leq q^{m(\frac{1}{2}-\alpha)} \leq \tfrac{1}{2}$
for large $m$. 
Moreover, it follows that 
\[
\nu:=(1/2 - \alpha)(\lfloor \log_q(1/4)\rfloor + 1 - \log_q(1/4)) > 0.
\]
Thus, applying Lemma~\ref{genupper} yields 
\begin{align*}
\lefteqn{\exp[\stab(\geq \tfrac{m}{2};\geq \lfloor \log_q(1/4) \rfloor + 1)]}\quad \\
\quad &\le  2^{m+1} \left(n q^{\frac{m}{2}}\right)^{\lfloor \log_q(1/4) \rfloor + 1} \\
&=  2 n^{\log_q(1/4)} 2^m n^{\lfloor \log_q(1/4) \rfloor + 1 - \log_q(1/4)}q^{\frac{\lfloor \log_q(1/4) \rfloor + 1}{2} m} \\
&\leq  2n^{\log_q(1/4)} q^{- \log_q(1/2) m - \alpha (\lfloor \log_q(1/4) \rfloor + 1 - \log_q(1/4)) m + \frac{\lfloor \log_q(1/4) \rfloor + 1}{2} m} \\
&=  2n^{\log_q(1/4)} q^{m (1/2 - \alpha)(\lfloor \log_q(1/4)\rfloor + 1 - \log_q(1/4))}\\
&\le  2n^{\log_q(1/4)} q^{\nu m}
\end{align*}
for sufficiently large $m$.
With Markov's inequality we deduce
\[
\pr[\stab(\geq \tfrac{m}{2};\geq \lfloor \log_q(1/4) \rfloor + 1) > n^{\log_q(1/4)}]
\leq \frac{2 n^{\log_q(1/4)} q^{\nu m}}{n^{\log_q(1/4)}} = 2q^{\nu m},
\]
which tends to~$0$ as $m\to\infty$. Since $n\leq q^{-\alpha m}$ 
implies that also $m$ must be large for large $m+n$, we may find an $N$
so that 
\(
\pr[\stab(\geq \tfrac{m}{2};\geq \lfloor \log_q(1/4) \rfloor + 1) > n^{\log_q(1/4)}]\leq\epsilon,
\) 
for all for all integers $m,n$ with $m+n\geq N$. 

Considering such $m$ and $n$, we turn now to the number of maximal stable sets 
$\lmsshalf$ with left side $\geq\tfrac{m}{2}$.
As in Lemma~\ref{largeleftupper}, we argue that the maximal stable sets counted 
by $\lmsshalf$ split into those whose right side have at 
least $\lfloor \log_q(1/4) \rfloor + 1$ vertices and those with at most $\lfloor \log_q(1/4) \rfloor$
vertices in $R(G)$. Of the latter ones, there are at most $n^{\log_q(1/4)}$ many sets. 
By choice of $N$, the probability that we have more than $n^{\log_q(1/4)}$ of the former is
bounded by $\epsilon$.
\end{proof}

Let us quickly calculate the probability that a given set of vertices is a maximal stable set.

\begin{lemma}\label{mssproba}
For a random bipartite graph $G\in\rbip(m,n;p)$, let
$S$ be a subset of $V(G)$.
If $|S\cap L|=\ell$  and $|S\cap R|=r$ then 
\[
\pr[S\in\mstab]=
 q^{\ell r}
\left(1-q^r\right)^{m-\ell}
\left(1-q^{\ell}\right)^{n-r}.
\]
\end{lemma}

\begin{proof}
The factor $q^{\ell r}$ is the probability that there is no edge from $S \cap L(G)$ to $S \cap R(G)$, that is, $S$ is a stable set.
The factor $\left(1-q^r\right)^{m-\ell}$ is the probability that every of the $m-\ell$ many vertices in $L(G) \setminus S$ has a neighbour in $S \cap R(G)$, and $\left(1-q^\ell\right)^{n-r}$ is the probability that every of the $n-r$ many vertices in $R(G) \setminus S$ has a neighbour in $S \cap L(G)$.
The latter two conditions ensure that $S$ is a maximal stable set.
\end{proof}

Next, we calculate the expectation and the variance of the number 
of maximal stable sets of small left side. 
Since they outnumber the other maximal stable sets by far, we concentrate on those
maximal stable sets with a left side equal to $\approx\log_{1/q}(n)$ and a right side equal to $\approx\log_{1/q}(m)$.
This choice is somewhat forced by the  maximality requirement 
for maximal stable sets: For logarithmic sized left sides  the 
maximality requirement eliminates only a constant proportion of the possible 
maximal stable sets. With smaller sides, on the other hand, we lose 
more sets, that is, the expectation becomes much smaller. 

For $G \in \rbip(m,n;p)$ 
 we  denote by $\smallmss$ the number of 
maximal stable sets $S$ of $G$ 
with $|S \cap L(G)| = \Lside$ and $|S \cap R(G)| = \Rside$.

\begin{lemma}\label{exp.small.mss}
Let $c = e^{-(2/q+1)}$.
There are $m_0,n_0 \in \mathbb{N}$ such that for $G\in\rbip(m,n;p)$
\[
\exp[\smallmss] \geq c  { m \choose \Lside} \, \Rside^{-\Rside}
\]
for all $m\geq m_0$, $n\geq n_0$ with $m\geq\log_{1/q}(n)$ and $n\geq\log_{1/q}(m)$.
\end{lemma}
\begin{proof}
Assume that $m,n$ are integers with $m\geq\log_{1/q}(n)$ and $n\geq\log_{1/q}(m)$.
We first note that 
\begin{align*}
\left( 1-q^{\Rside}\right)^{m-\Lside} &\geq \left( 1-q^{\Rside}\right)^{m}\\
&\geq \left( 1-q^{\log_{1/q}(m) - 1}\right)^{m}\\
&= \left(1-\frac{1}{qm}\right)^{m}.
\end{align*}
Since $\lim_{m \rightarrow \infty}\left(1-\frac{1}{qm}\right)^{m} = e^{-1/q}$, there is $m_0 \in \mathbb{N}$ such that 
\begin{equation}
\label{onebound}\left(1-\frac{1}{qm}\right)^{m} \ge e^{-\left(\tfrac{1}{q}+\tfrac{1}{2}\right)},
\end{equation}
for all $m \ge m_0$. With the same arguments, we see that there is an $n_0\in\mathbb N$, 
so that 
\[
\left( 1-q^{\Lside}\right)^{n-\Rside}\geq e^{-\left(\tfrac{1}{q}+\tfrac{1}{2}\right)}.
\]
for all $n \geq n_0$.

Now consider a random bipartite graph $G \in \rbip(m,n;p)$ with $m\geq m_0$
and $n\geq n_0$, and let $S$ be any vertex subset with $|S\cap L(G)|=\Lside=:a$ and 
$|S\cap R(G)|=\Rside=:b$.
By Lemma~\ref{mssproba}, the probability that $S$ 
is a maximal stable set of $G$ amounts to
\[
\pr[S\text{\rm\ is maximally stable}] = q^{ab}\left( 1-q^{b}\right)^{m-a}
\left( 1-q^{a}\right)^{n-b}.
\]
The first term is at least equal to $n^{-\Rside}$, while, by~\eqref{onebound},
the remaining terms together are at least equal to $c=e^{-(2/q+1)}$.
This yields 
\[
\pr[S \text{\rm\ is maximally stable}] \geq cn^{-\Rside}.
\]
Thus,
\begin{align*}
\exp[\smallmss] &\ge {m \choose \Lside} {n \choose \Rside} c n^{-\Rside} \\
&\ge {m \choose \Lside} 
\left( \frac{n}{\Rside} \right)^{\Rside} c n^{-\Rside} \\
&= c {m \choose \Lside}\, \Rside^{-\Rside}.
\end{align*}
\end{proof}

We use Chebyshev's inequality to show that, with high probability, $\smallmss$ 
does not differ much from the expected value.

\begin{lemma}\label{superpoly.lower.bound}
For every $\epsilon>0$  
there is an $N$ so that for $G\in\rbip(m,n;p)$  
\[
\pr\left[\smallmss > \tfrac{1}{2} \exp[\smallmss]\right] \geq 1-\epsilon,
\]
for all $m,n$ with $m+n\geq N$, $m\leq q^{-\sqrt[5]{n}}$ and 
$n\leq q^{-\sqrt[5]{m}}$.
\end{lemma}

\begin{proof}
Since the statement of Lemma~\ref{superpoly.lower.bound} is symmetric in $m$ and $n$, we may assume throughout the proof that $m \le n$.
Moreover we assume that $m\leq q^{-\sqrt[5]{n}}$ and 
$n\leq q^{-\sqrt[5]{m}}$.

Let $a := \Lside$ and $b := \Rside$.
Chebyshev's inequality \eqref{cheb} gives us
\[
\pr\left[\smallmss \le \tfrac{1}{2} \exp[\smallmss]\right] \le \frac{4\sigma^2}{ \exp[\smallmss]^2},
\]
where $\sigma^2 = \exp[\smallmss^2] - \exp[\smallmss]^2$ is the variance of the random variable $\smallmss$.
We have
\[
\exp[\smallmss^2] = \sum_{i=0}^a \sum_{j=0}^b A_{i,j},
\]
where $A_{i,j}$ denotes the expected number of pairs $(S,T)$ of maximal stable sets of $G$ 
with $|S \cap L(G)| = a = |T \cap L(G)|$, $|S \cap R(G)| = b = |T \cap R(G)|$, 
$|S \cap T \cap L(G)| = i$, and $|S \cap T \cap R(G)| = j$.
By Lemma~\ref{mssproba},
\begin{equation}\label{exp.lrside}
\exp[\smallmss]^2 = {m \choose a}^2 {n \choose b}^2 q^{2ab} (1-q^a)^{2(n-b)} (1-q^b)^{2(n-a)}.
\end{equation}

We will first show that 
there is an $N_1$ so that 
\begin{equation}\label{lim.A00}
\frac{A_{0,0} -\exp[\smallmss]^2}{ \exp[\smallmss]^2} \leq \frac{\epsilon}{8},
\end{equation} 
for all $m,n$ with $m+n\geq N_1$.

To prove this, observe that
\[
A_{0,0} \leq {m \choose a} {m-a \choose a} {n \choose b} {n-b \choose b} q^{2ab} (1-q^a)^{2(n-2b)} (1-q^b)^{2(m-2a)}.
\]
Indeed, while the binomial coefficients count the number of possibilities to choose 
the disjoint sets $S$ and $T$, the factor $q^{2ab}$ is the probability that $S$ and $T$ are stable sets.
Furthermore, the probability that every vertex in $R(G)\sm (S\cup T)$ has a neighbour in $S$ and a 
neighbour in $T$ is  equal to $(1-q^a)^{2(n-2b)}$; the factor $(1-q^b)^{2(m-2a)}$
expresses the analogous probability for $L(G)$. 

The above estimation for $A_{0,0}$ together with \eqref{exp.lrside} yields
$A_{0,0} / \exp[\smallmss]^2 \leq (1-q^a)^{-2b} (1-q^b)^{-2a}$, and consequently
\[
\frac{A_{0,0} -\exp[\smallmss]^2}{\exp[\smallmss]^2} \leq (1-q^a)^{-2b} (1-q^b)^{-2a} - 1.
\]
Next, note that if $n$ is large enough so that $\tfrac{1}{nq}\leq\tfrac{1}{2}$ then
\begin{align*}
(1-q^a)^{2b} & \geq (1-q^{\log_{1/q}(n)-1})^{2\log_{1/q}(m)} \geq \left(1-\frac{1}{nq}\right)^{2\sqrt[5]{n}}\\
&\geq \left(e^{-\frac{2}{nq}}\right)^{2\sqrt[5]{n}}\to e^0=1\text{ as }n\to \infty, 
\end{align*}
where we have used that $1-x\geq e^{-2x}$ for all $0\leq x\leq 1/2$.
The analogous estimation holds for $(1-q^b)^{-2a}$. Thus, if $m$ and $n$ are large enough, 
then 
\[
(1-q^a)^{-2b} (1-q^b)^{-2a} - 1\leq \frac{\epsilon}{8}.
\]
Since it follows from $m\leq q^{-\sqrt[5]{n}}$ and 
$n\leq q^{-\sqrt[5]{m}}$ that both of $m$ and $n$ have to be large if $m+n$ is large, 
we may therefore choose  $N_1$ so that~\eqref{lim.A00} holds. 

\medskip
We will now investigate $A_{i,j} / \exp[\smallmss]^2$ when $i+j \ge 1$.
For this, let
\begin{equation}\label{Bij}
B_{i,j} = {m \choose i} {m-i \choose a-i} {m-a \choose a-i} {n \choose j} {n-j \choose b-j} {n-b \choose b-j} q^{2ab - ij}.
\end{equation}
Note that $B_{i,j}$ equals the expected number of pairs $(S,T)$ of stable sets
(not necessarily maximal)
 with $|S \cap L(G)| = a = |T \cap L(G)|$, $|S \cap R(G)| = b = |T \cap R(G)|$, $|S \cap T \cap L(G)| = i$, and $|S \cap T \cap R(G)| = j$.
Hence, $A_{i,j} \le B_{i,j}$ for all $0 \le i \le a$ and $0 \le j \le b$.

For $r,s \in \mathbb{N}$ with $r \ge s$, let $(r)_s$ 
denote the $s$-th falling factorial of $r$, i.e., $(r)_s = r (r-1) \cdots (r-s+1)$.
For the binomial coefficients appearing in $B_{i,j}/\exp[\smallmss]^2$ 
that involve $m$, we deduce
\begin{align*}
\frac{{m \choose i} {m-i \choose a-i} {m-a \choose a-i}}
{{m\choose a}^2} 
= \frac{(m-a)_{a-i}\,(a)_i^2}
{(m)_a\, i!} 
\leq \frac{(a)_i^2}
{(m)_i\, i!}
\leq \frac{(a)_i^2}
{m^i}
\leq \frac{a^{2i}}
{m^i},
\end{align*}
for all integers $i$ with $0 \le i \le a$. With the analogous estimation 
for the binomial coefficients
involving $n$, we obtain
\begin{equation}\label{binest}
\frac{{m \choose i} {m-i \choose a-i} {m-a \choose a-i}
{n \choose j} {n-j \choose b-j} {n-b \choose b-j}}
{{m\choose a}^2 {n\choose b}^2}\leq 
\frac{a^{2i} b^{2j}}
{m^in^j}, 
\end{equation}
for all integers $i,j$ with $0 \le i \le a$ and $0 \le j \le b$.

An easy calculation 
(very similar to \eqref{onebound}) shows that there is a constant $c$ such that 
\begin{equation}\label{yetanother}
c \ge 4(1-q^a)^{-2(n-b)} (1-q^b)^{-2(m-a)} 
\end{equation}
for all $m$.

Recalling the explicit expression~\eqref{exp.lrside} for $\exp[\smallmss]^2$, 
and then applying first~\eqref{binest} and then~\eqref{yetanother}
we deduce  
\begin{align*}
\frac{(a+1) (b+1) B_{i,j}}{\exp[\smallmss]^2} 
&= \frac{(a+1) (b+1){m \choose i} {m-i \choose a-i} {m-a \choose a-i} {n \choose j} {n-j \choose b-j} {n-b \choose b-j} q^{2ab - ij}}
{{m \choose a}^2 {n \choose b}^2 q^{2ab} (1-q^a)^{2(n-b)} (1-q^b)^{2(m-a)}}\\
&\stackrel{\eqref{binest}}{\le} \frac{ (a+1) (b+1) a^{2i} b^{2j}}{m^i n^j q^{ij}
(1-q^a)^{2(n-b)} (1-q^b)^{2(m-a)}}\\
&\le \frac{4 a^{2i+1} b^{2j+1}}{m^i n^j q^{ij}
(1-q^a)^{2(n-b)} (1-q^b)^{2(m-a)}}\\
&\stackrel{\eqref{yetanother}}{\le} \frac{c a^{2i+1} b^{2j+1}}{m^i n^j q^{ij}}
\end{align*}
for all $i,j$ with $i+j \ge 1$, and where we assume in the third 
step that $m$ and $n$ are large enough so that $a=\lfloor\log_{1/q}(n)\rfloor\geq 1$ 
and $b=\lfloor\log_{1/q}(m)\rfloor\geq 1$.
(Again, this is possible since $m\leq q^{-\sqrt[5]{n}}$ and 
$n\leq q^{-\sqrt[5]{m}}$ implies that both of $m$ and $n$ have to be large if $m+n$ is large.)

In order to continue with the estimation we consider the term $m^i n^j q^{ij}$. 
For $0 \leq i \leq a$ and $0 \le j \le b$, 
we see that 
\(
m^iq^{ij}\geq (mq^b)^i\geq (mq^{\log_{1/q}(m)})^i=(\frac{m}{m})^i=1.
\)
In a similar way, we obtain $n^j q^{ij} \ge 1$. 
Now, if $i\geq j$ then $m^i n^j q^{ij}= m^i(n^jq^{ij})\geq m^i$.
If, on the other hand, $i<j$ then $n^jm^iq^{ij}\geq n^j\geq m^j$, 
since $m\leq n$.
Thus
\begin{equation}\label{sparem}
m^i n^j q^{ij} \geq m^{\max(i,j)}, \text{ for all $m,n$ with $m\leq n$}.
\end{equation}
Using~\eqref{sparem}, we obtain
\begin{align*}
\frac{(a+1) (b+1) B_{i,j}}{\exp[\smallmss]^2} 
&\leq 
 \frac{c a^{2i+1} b^{2j+1}}{m^i n^j q^{ij}}
\le \frac{c a^{2i+1} b^{2j+1}}{m^{\max(i,j)}}\\
&\leq \frac{c  \log_{1/q}(n)^{2i+1} \log_{1/q}(m)^{2j+1}}{m^{\max(i,j)}}\\
&\leq \frac{c  (\sqrt[5]{m})^{2i+1} \log_{1/q}(m)^{2j+1}}{m^{\max(i,j)}}\\
&\leq \frac{c  m^{\frac{2}{5}i+\frac{1}{5}} \log_{1/q}(m)^{2j+1}}{m^{\max(i,j)}},
\end{align*}
since $n\leq q^{-\sqrt[5]{m}}$ and $m\leq n$. 
Now, since $i+j\geq 1$, the last term tends to~$0$ for $m\to\infty$. Therefore, 
there is an $N_2$ independent of $i,j$ such that for all $m+n\geq N_2$
\[
\frac{(a+1) (b+1) B_{i,j}}{\exp[\smallmss]^2} 
\leq \frac{\epsilon}{8},
\]
whenever $i+j\geq 1$.

Thus, for $N=\max(N_1,N_2)$ and $m,n$ with $m+n\geq N$,
we get with~\eqref{lim.A00}
\begin{align*}
\pr[\smallmss \le \tfrac{1}{2} \exp[\smallmss]] &\le \frac{4\sigma^2}{ \exp[\smallmss]^2} \\
&\le 4\frac{A_{0,0}-\exp[\smallmss]^2}{ \exp[\smallmss]^2} \\
&+ \frac{4(a+1) (b+1) \max\{B_{i,j} : 0 \le i \le a, 0 \le j \le b, i+j \ge 1 \}}{ \exp[\smallmss]^2}\\
&\le 4\frac{\epsilon}{8} + 4\frac{\epsilon}{8} = \epsilon.
\end{align*}
\end{proof}

Let us quickly explain why the method of Lemma~\ref{superpoly.lower.bound} 
ceases to work when $n \geq q^{-\sqrt{m}}$.
In the proof we aim for $\pr[ \smallmss < \tfrac{1}{2} \exp[\smallmss]]$ to tend to~$0$ 
with growing $m+n$. We achieve that by forcing  the right hand side of \eqref{binest} 
to vanish for large $m+n$, and all $i,j$ with $i + j \ge 1$.
In particular, when $i =1$ and $j=0$, we need $\tfrac{a^2}{m} = {\Lside}^2 / m$ 
to vanish with growing $m$, which in turn requires that $n < q^{-\sqrt{m}}$.
\medskip

To finish this case, we observe that, 
with high probability, Lemma~\ref{squpperbound} bounds the number $\lmss$
of maximal stable sets with large left  side with a polynomial in~$n$, while we will see below
that, again with high probability,
Lemmas~\ref{exp.small.mss} and~\ref{superpoly.lower.bound} translate into 
superpolynomially many maximal stable sets with small left side.
\begin{lemma}\label{mdmsides}
For every $\epsilon>0$ there is an $N$ so that 
for every $G\in \rbip(m,n;p)$
\[
\pr[\av\leq\tfrac{m}{2}] \geq 1-\epsilon
\]
for all $m,n$ with $m+n\geq N$,  $m \leq q^{-\sqrt[5]{n}}$ and $n \leq q^{-\sqrt[5]{m}}$.
\end{lemma}

\begin{proof}
Choose $N_1$ large enough so that it is at least as large as the $N$ in
Lemma~\ref{exp.small.mss} with $\tfrac{\epsilon}{4}$ and as the $N$ in 
Lemma~\ref{superpoly.lower.bound}, as well with $\tfrac{\epsilon}{4}$, 
and so that $\lfloor\log_{1/q}(n)\rfloor\geq\tfrac{1}{2}\log_{1/q}(n)$
for any $m,n$ with $m+n\geq N_1$ and $m \leq q^{-\sqrt[5]{n}}$.

In the remainder of the proof, we consider integers $m,n$ with $m+n\geq N_1$,  $m \leq q^{-\sqrt[5]{n}}$ and $n \leq q^{-\sqrt[5]{m}}$.
By Lemmas~\ref{exp.small.mss} and~\ref{superpoly.lower.bound},
there is a constant $c>0$, independent of $m$ and $n$, so that 
the probability that  
\begin{align*}
\smallmss 
& \leq c \left( \frac{m}{\Lside} \right)^{\Lside}\Rside^{-\Rside}\\
&\leq c {m\choose \Lside}\, \Rside^{-\Rside},
\end{align*}
is at most $\tfrac{\epsilon}{2}$. 

Now, using that $\log_{1/q}(m)\leq\sqrt[5]{n}$ and $\log_{1/q}(n)\leq\sqrt[5]{m}$
we obtain
\begin{align*}
& c \left( \frac{m}{\Lside} \right)^{\Lside} \Rside^{-\Rside} \\
& \geq c m^{\frac{1}{2}\log_{1/q}(n)}\cdot \log_{1/q}(n)^{-\log_{1/q}(n)}\cdot \log_{1/q}(m)^{-\log_{1/q}(m)}\\
& \geq cn^{\frac{1}{2}\log_{1/q}(m)}\cdot n^{-\log_{1/q}(\log_{1/q}(n))}\cdot \left(\sqrt[5]{n}\right)^{-\log_{1/q}(m)}\\
& \geq cn^{\frac{1}{2}\log_{1/q}(m)}\cdot n^{-\log_{1/q}(\sqrt[5]{m})}\cdot \left(n\right)^{-\frac{1}{5}\log_{1/q}(m)}
= cn^{\frac{1}{10}\log_{1/q}(m)}.
\end{align*}
It follows that 
\begin{equation}\label{stepA}
\pr\left[\smallmss\leq cn^{\frac{1}{10}\log_{1/q}(m)}\right]\leq\frac{\epsilon}{2}.
\end{equation}

On the other hand, we obtain from Lemma~\ref{squpperbound} that there is a $N_2$ 
so that 
\begin{equation}\label{stepB}
\pr\left[\lmsshalf >2n^{\log_{q}(1/4)}\right] \leq \frac{\epsilon}{2},
\end{equation}
whenever $m+n\geq N_2$.

Recall that $\smallmss$ is a lower bound on the number of maximal stable sets $S$
with $|S\cap L(G)|=\lfloor\log_{1/q}(n)\rfloor$. 
Since $m\leq q^{-\sqrt[5]{n}}$ and 
$n\leq q^{-\sqrt[5]{m}}$ implies that both of $m$ and $n$ have to be large if $m+n$ is large,
we may choose $N_3$ large enough so that 
 $\sqrt[5]{m}\leq \tfrac{m}{4}$ 
and $cn^{\frac{1}{10}\log_{1/q}(m)} \geq 4n^{\log_{q}(1/4)}$ whenever $m+n\geq N_3$.
We claim that 
\begin{equation}\label{stepC}
\text{\em
if $\smallmss\geq cn^{\frac{1}{10}\log_{1/q}(m)}$ and $\lmsshalf\leq 2n^{\log_{q}(1/4)}$
then $\av\leq\tfrac{m}{2}$. 
}
\end{equation}
Indeed, since $\lfloor\log_{1/q}(n)\rfloor\leq\sqrt[5]{m}\leq\tfrac{m}{4}$ 
this is a direct consequence of Lemma~\ref{largeandsmall} with $\nu=\tfrac{1}{2}$
and $\delta=0$.

Finally, the lemma follows from~\eqref{stepA},~\eqref{stepB} and~\eqref{stepC}
if $N$ is chosen to be at least $\max(N_1,N_2,N_3)$.
\end{proof}

\subsection{The case $q^{-\sqrt[5]{m}} \le n \le q^{-\frac{m}{16}}$}\label{sect.m/16}

When the right side of the random bipartite graph becomes much larger than the left side,
we cannot control the variance of $\smallmss$ anymore, and indeed Chebyshev's 
inequality cannot even give a positive probability, however small, that the graph 
contains many maximal stable sets of small left side. 
We therefore use another standard tool, Hoeffding's 
inequality, which yields a tiny but non-zero probability. 
We then leverage this tiny probability
to a high probability by cutting up the right side of the graph into a large number
of large pieces to each of which we apply Hoeffding's theorem:

\begin{theorem}[Hoeffding~\cite{Hoe63}]\label{thm.hoeffding}
For $i=1,\ldots,s$, let $X_i:\Omega_i\to [0,\rho]$ be independent random variables, 
and let $X=\sum_{i=1}^sX_i$. Then 
\[
\pr[X\geq \exp[X]+\lambda]\leq e^{-\frac{2\lambda^2}{s\rho^2}}
\]
and 
\[
\pr[X\leq \exp[X]-\lambda]\leq e^{-\frac{2\lambda^2}{s\rho^2}}
\]
for every $\lambda $ with $\lambda>0$.
\end{theorem}

We will use Theorem~\ref{thm.hoeffding} in the simpler case, when there is 
only one random variable, that is, $s$ will be equal to~$1$.

By $\hoeffmss$ we  denote the number of maximal stable sets $S$ of $G$ with $|S \cap L(G)| = \lfloor \log_{1/q}(\lfloor n / m^{\log_{1/q}(m)} \rfloor) \rfloor$.
Note that, in contrast to $\smallmss$, we put no restriction on $|S \cap R(G)|$.

\begin{lemma}\label{lem.hoeffding.exp}
For integers $m,n$ let $a':=\lfloor \log_{1/q}(\lfloor n / m^{\log_{1/q}(m)} \rfloor) \rfloor$
and $b:=\Rside$. Then there exists a constant $c>0$ so that for every $\epsilon>0$
there is an $N$ such that for $G\in\rbip(m,n;p)$ 
\[
\pr\left[\hoeffmss \ge c{m\choose a'}b^{-b}\right] \ge 1-\epsilon,
\]
for all $m,n$ with $m+n\geq N$ and $m^{2 \log_{1/q}(m)} \leq n \leq q^{-m}$.
\end{lemma}
\begin{proof}
In the following we assume that $m^{2 \log_{1/q}(m)} \leq n \leq q^{-m}$.
Let $k = m^{\log_{1/q}(m)}$, and let $R_1, R_2, \ldots, R_{\lfloor k \rfloor}$ be disjoint subsets of $R(G)$ of size $\lfloor n/k \rfloor$ each.
Moreover, let $G_i = G[L(G) \cup R_i]$ for all $1 \le i \le \lfloor k \rfloor$.
Note that every $G_i$ may be considered as a random bipartite graph in $\rbip(m,\lfloor n/k \rfloor;p)$.

Let $a' = \lfloor \log_{1/q}(\lfloor n/k \rfloor) \rfloor$ and $b = \Rside$.
Note that, since $m^{2 \log_{1/q}(m)} \leq n \leq q^{-m}$, also $m \ge a'$ and $n \ge b$.
Recall that, for $1 \le i \le k$, $\mathcal{S}_{G_i}$ is the number of maximal stable sets of $G_i$ with $|\mathcal{S}_{G_i} \cap L(G)| = a'$ and $|\mathcal{S}_{G_i} \cap R(G)| = b$.

Applying Lemma~\ref{exp.small.mss} to $G_i$, we get
\begin{equation}\label{Yi.lower.bound}
\exp[\mathcal{S}_{G_i}] \ge c {m \choose a'} b^{-b},
\end{equation}
where $c > 0$ is some constant.

Clearly the value of $\mathcal{S}_{G_i}$ does never exceed ${m \choose a'}$.
Thus, by Theorem~\ref{thm.hoeffding},
\[
\pr[\mathcal{S}_{G_i} \le \tfrac{1}{2} \exp[\mathcal{S}_{G_i}]] \le e^{- \frac{1}{2} \mathbf{E}[\mathcal{S}_{G_i}]^2 {m \choose a'}^{-2}} \stackrel{\eqref{Yi.lower.bound}}{\le} e^{- \frac{1}{2} c^2 b^{-2b}}.
\]
The edge sets of the $G_i$ are pairwise disjoint, and thus the random variables $\mathcal{S}_{G_i}$
are independent:
\[
\pr[\mathcal{S}_{G_i} \le \tfrac{1}{2} \mathbf{E}[\mathcal{S}_{G_i}] \mbox{ for all } 1 \le i \le \lfloor k \rfloor] \le e^{- \lfloor k \rfloor \frac{1}{2} c^2 b^{-2b}}.
\]

Since $\lfloor k \rfloor = \lfloor m^{\log_{1/q}(m)} \rfloor$, it dominates $b^{-2b}$.
Hence, $\lim_{m \rightarrow \infty} \lfloor k \rfloor \tfrac{1}{2} c^2 b^{-2b} = \infty$.
Thus, there is an $N$ such that $e^{- \lfloor k \rfloor \frac{1}{2} c^2 b^{-2b}} \le \epsilon$ whenever $m + n \ge N$, as $m^{2 \log_{1/q}(m)} \leq n \leq q^{-m}$ implies that $m$ grows with $m+n$.
Hence, assuming $m + n \ge N$,
\begin{equation}\label{hoeffding.trick}
\pr[\mathcal{S}_{G_i} \le \tfrac{1}{2} \mathbf{E}[\mathcal{S}_{G_i}] \mbox{ for all } 1 \le i \le \lfloor k \rfloor] \le e^{- \lfloor k \rfloor \frac{1}{2} c^2 b^{-2b}} \le \epsilon,
\end{equation}
Thus, with probability $1-\epsilon$, there is an $i$ for which $\mathcal{S}_{G_i} \ge c {m \choose a'} b^{-b}$.
Note that every maximal stable set $S$ of $G_i$ can be extended to a maximal stable set $S'$ of $G$ such that $S' \cap V(G_i) = S$.
This extension is injective and, moreover, $|S' \cap L(G)| = |S \cap L(G)| = a'$.
Hence, every maximal stable set counted by $\mathcal{S}_{G_i}$ is also counted by $\hoeffmss$, i.e., $\hoeffmss \ge \mathcal{S}_{G_i}$.
This completes the proof.
\end{proof}

Observe that, in order to apply \eqref{hoeffding.trick}, 
we need $k$ to dominate $b^{2b}$, where $b=\Rside$.
Hence, $k$ and thus also $n$ should be of the order at least $m^{2 \log_{1/q}(m) \log_{1/q}(\log_{1/q}(m))}$.
This means that we could not use this method before, when $m$ and $n$ had about the same size.

\begin{lemma}\label{largerightside}
For every $\epsilon>0$ there is an $N$ so that for $G\in\rbip(m,n;p)$ 
\[
\pr\left[\av\leq \tfrac{m}{2}\right]\geq 1-\epsilon,
\]
for all $m,n$ with $m+n\geq N$ and $q^{-\sqrt[5]{m}} \le n \le q^{-\frac{m}{16}}$.
\end{lemma}
\begin{proof}
In this proof consider integers $m,n$ with $q^{-\sqrt[5]{m}} \le n \le q^{-\frac{m}{16}}$,
and let $a':=\lfloor \log_{1/q}(\lfloor n / m^{\log_{1/q}(m)} \rfloor) \rfloor$.
Note that $a' \ge \log_{1/q}(n) - \log_{1/q}(m^{\log_{1/q}(m)}) - 2$.
Then
\begin{align*}
{m \choose a'} &\ge \left( \frac{m}{a'} \right)^{a'} 
\ge \left( \frac{m}{\log_{1/q}(n)} \right)^{a'}\\ 
&\ge \left( \frac{m}{\log_{1/q}(n)} \right)^{\log_{1/q}(n) - \log_{1/q}(m^{\log_{1/q}(m)}) - 2}\\ 
&\geq m^{\log_{1/q}(n)- (\log_{1/q}(m)^2 + 2)}\cdot \left(\log_{1/q}(n)\right)^{-\log_{1/q}(n)}\\
&= n^{\log_{1/q}(m)}\cdot m^{- (\log_{1/q}(m)^2 + 2)}\cdot n^{-\log_{1/q}(\log_{1/q}(n))}\\
&\geq n^{\log_{1/q}(m)}\cdot m^{- (\log_{1/q}(m)^2 + 2)}\cdot n^{-\log_{1/q}(m/16)}\\
& = n^{\log_{1/q}(16)}\cdot m^{- (\log_{1/q}(m)^2 + 2)} = n^{\log_{1/q}(8)}\cdot n^{\log_{1/q}(2)}\cdot m^{- (\log_{1/q}(m)^2 + 2)}
\end{align*}
In Lemma~\ref{lem.hoeffding.exp} the binomial coefficient ${m\choose a'}$ is 
divided by $\lfloor\log_{1/q}(m)\rfloor^{\lfloor\log_{1/q}(m)\rfloor}$. 
So, let us compare this factor times $m^{\log_{1/q}(m)^2 + 2}$ against $n^{\log_{1/q}(2)}$. 
When $m$ is large enough, which we may assume since $m^{2 \log_{1/q}(m)} \leq n \leq q^{-m}$ implies that $m$ grows with $m+n$, we get that $(\log_{1/q}(m))^2\geq 2+\log_{1/q}(\log_{1/q}(m))$. 
Thus
\begin{align*}
m^{\log_{1/q}(m)^2 + 2}\cdot \lfloor\log_{1/q}(m)\rfloor^{\lfloor\log_{1/q}(m)\rfloor}
& \leq m^{\log_{1/q}(m)^2 + 2}\cdot (\log_{1/q}(m))^{\log_{1/q}(m)}\\
&= m^{\log_{1/q}(m)^2 + 2+\log_{1/q}(\log_{1/q}(m))}\\
&\leq m^{2\log_{1/q}(m)^2}.
\end{align*}

Using $q^{-\sqrt[5]{m}} \le n$, we get
\begin{align*}
m^{2\log_{1/q}(m)^2} & \leq \left((\log_{1/q}(n))^5\right)^{2(\log_{1/q}(\log_{1/q}(n)^5))^2}\\
& = \left(\log_{1/q}(n)\right)^{250(\log_{1/q}(\log_{1/q}(n))^2}\\
& = q^{-250(\log_{1/q}(\log_{1/q}(n))^3}.
\end{align*}
Since $\log_{1/q}(2)>0$ and $n^{\log_{1/q}(2)}=q^{-\log_{1/q}(2)\cdot \log_{1/q}(n)}$, 
we see that $n^{\log_{1/q}(2)}>m^{2\log_{1/q}(m)^2}$ for large enough $m$ and $n$.

In conjunction with Lemma~\ref{lem.hoeffding.exp} this yields that there is a constant $c>0 $ and an $N_1$ so that 
\begin{equation}\label{win.less.than.m/16}
\pr[\hoeffmss \ge c n^{\log_{1/q}(8)}] \ge 1 - \frac{\epsilon}{2},
\end{equation}
whenever $m+n\geq N_1$.

On the other hand, Lemma~\ref{squpperbound} yields an $N_2$ so that 
\[
\pr[\lmsshalf>2n^{\log_{q}(1/4)}] \leq \frac{\epsilon}{2},
\]
when $m+n\geq N_2$.

Now we choose an $N\geq\max(N_1,N_2)$ so that $2n^{\log_{q}(1/4)}= 2n^{\log_{1/q}(4)}$
is much smaller than $c n^{\log_{1/q}(8)}$, by a factor of $2$, say,  
when $m+n\geq N$.
(This is possible as $m^{2 \log_{1/q}(m)} \leq n \leq q^{-m}$ implies that $m$ is large when $m+n$ is large.)
Thus $2\hoeffmss\geq \lmsshalf$ with a probability of $\geq 1-\epsilon$, and Lemma~\ref{largeandsmall} (with $\delta=0$) completes the proof. Indeed, note that $\hoeffmss$ counts the number of maximal stable sets whose left sides have size 
$\lfloor \log_{1/q}(\lfloor n / m^{\log_{1/q}(m)} \rfloor) \rfloor\leq\log_{1/q}(n)\leq \tfrac{m}{16}$.
\end{proof}

Note that the estimation leading to \eqref{win.less.than.m/16} ceases to work 
when $n > q^{-\frac{m}{4}}$.
We need therefore a finer estimation, which is what we do in the next section.

\subsection{The case $q^{-\frac{m}{16}} \le n < q^{-\frac{m}{2}}$}\label{sect.alpha.m}

For $\kappa \in (0,1)$ the \emph{binary entropy} is defined as
\[
H(\kappa)=\kappa\log_2(\tfrac{1}{\kappa})+(1-\kappa)\log_2(\tfrac{1}{1-\kappa}).
\]
Observe that the binary entropy $H(\kappa)$ is always strictly smaller than~$1$,
except for $\kappa=\tfrac{1}{2}$. 
Moreover, $H$ is monotonously increasing in the interval $\left[0,\tfrac{1}{2}\right]$.
For further details see \cite{MUBook05}.

We will use the following bound on the binomial coefficient, which can be found for instance in Mitzenmacher and Upfal~\cite[Lemma~9.2]{MUBook05}.

\begin{lemma}\label{improvedbin}
For all $m,k \in \mathbb{N}$ with $0 < k < m$,
\[
{m \choose k}\geq
\frac{1}{m+1}
\cdot 2^{H(k/m)\cdot m}.
\]
\end{lemma}


\begin{lemma}\label{asymptotic.lower.bound}
For integers $m,n$ let $\lambda := \log_{1/q}(n) / m$.
For every $\epsilon, \varphi > 0$ there is an $N$ such that for $G \in \rbip(m,n;p)$,
\[
\pr[\hoeffmss \geq 2^{(1-\varphi) \cdot H(\lambda) \cdot m}] \ge 1-\epsilon,
\]
for all $m,n$ with $m+n\geq N$ and $q^{-\frac{m}{16}} \le n \le q^{-\frac{m}{2}}$.
\end{lemma}

\begin{proof}
Throughout the proof assume $q^{-\frac{m}{16}} \le n \le q^{-\frac{m}{2}}$.

Let $\varphi > 0$.
First we choose a $\delta$ with $0 < \delta < 1$ which satisfies
\begin{equation}
H((1-\delta) \kappa) \ge \left( 1-\frac{\varphi}{2} \right) \cdot H(\kappa) \label{choice.of.delta}
\end{equation}
for all $\kappa \in \left[\tfrac{1}{16},\tfrac{1}{2}\right]$.
This is possible since $H$ is uniformly continuous in $\left[\tfrac{1}{16},\tfrac{1}{2}\right]$ 
and $\min \{H(\kappa) : \kappa \in \left[\tfrac{1}{16},\tfrac{1}{2}\right]\} > 0$.

Let $a' := \lfloor \log_{1/q}(\lfloor n / m^{\log_{1/q}(m)} \rfloor) \rfloor$.
Let $N_1$ be such that when $m+n \ge N_1$
\begin{equation}\label{choice.of.m0}
\lceil (1-\delta) \log_{1/q}(n) \rceil \le a' \le \left\lfloor \tfrac{m}{2} \right\rfloor.
\end{equation}
The choice of $N_1$ is possible since $\delta>0$ and $\log_{1/q}(n) \le \tfrac{m}{2}$.
In the following, we restrict our attention to these $m,n$ with $m+n \ge N_1$.
From \eqref{choice.of.m0} it follows that
\begin{equation}\label{a'.is.bigger}
{m \choose a'} \ge {m \choose \lceil (1-\delta) \log_{1/q}(n) \rceil}.
\end{equation}
Lemma~\ref{improvedbin} gives
\begin{equation}
{m \choose \lceil (1-\delta) \log_{1/q}(n) \rceil} \geq \frac{1}{m+1} \cdot 2^{H(\lceil (1-\delta) \log_{1/q}(n) \rceil m^{-1}) \cdot m}. \label{appl.better.est}
\end{equation}
Since $H(\kappa)$ is monotonically increasing 
for $\kappa \in \left[0,\tfrac{1}{2}\right]$, 
it follows that
$H(\lceil (1-\delta) \log_{1/q}(n) \rceil m^{-1}) \geq H( (1-\delta) \lambda)$, 
where we recall that $\lambda = \log_{1/q}(n) / m$.
As $\tfrac{1}{16}\leq\lambda \leq \tfrac{1}{2}$, 
we get
\begin{equation}
2^{H(\lceil (1-\delta) \log_{1/q}(n) \rceil m^{-1}) \cdot m} 
\geq 2^{H( (1-\delta) \lambda) \cdot m}
\stackrel{\eqref{choice.of.delta}}{\geq} 2^{\left( 1-\frac{\varphi}{2} \right) H(\lambda) \cdot m}. \label{finally1}
\end{equation}
Lemma~\ref{lem.hoeffding.exp} gives an $N_2$ and a constant $c>0$ such that
for $b = \Rside$
\[
\pr\left[\hoeffmss \ge c {m\choose a'} b^{-b} \right] \ge 1-\epsilon,
\]
when $m+n\geq N_2$.
Since $q^{-m/16} \le n \le q^{-m/2}$, there is an $N_3$ such that $m+n \ge N_3$ implies
\[
\frac{c}{m+1} \cdot b^{-b} \cdot 2^{\frac{\varphi}{2} \cdot H(\lambda) \cdot m} \ge 1,
\]
where we use that $\lambda\geq\tfrac{1}{16}$.
For such $m$ and $n$,
\begin{equation}\label{finally2}
\frac{c}{m+1} \cdot b^{-b} \cdot 2^{(1-\frac{\varphi}{2}) \cdot H(\lambda) \cdot m} 
\ge 2^{(1-\varphi) \cdot H(\lambda) \cdot m}.
\end{equation}

Now, taking $N = \max (N_1 , N_2, N_3)$, the inequalities \eqref{a'.is.bigger}, \eqref{appl.better.est}, \eqref{finally1} and \eqref{finally2} finish the proof.
\end{proof}

\begin{lemma}\label{hugerightside}
For every $\alpha$ with $\tfrac{1}{16}\leq\alpha<\tfrac{1}{2}$ and every $\epsilon>0$
there is an $N$ so that for $G\in\rbip(m,n;p)$ 
\[
\pr\left[\av\leq\tfrac{m}{2}\right]\geq 1-\epsilon,
\]
for all $m,n$ with $m+n\geq N$ and $q^{-\frac{m}{16}} \le n \le q^{-\alpha m}$.
\end{lemma}

\begin{proof}
Let $\tfrac{1}{16}\leq\alpha<\tfrac{1}{2}$ be given, and assume $q^{-\frac{m}{16}} \le n \le q^{-\alpha m}$.

Note that $2 \kappa < H(\kappa)$ for all $\tfrac{1}{16} \le \kappa \le \alpha$.
Since $H$ is continuous on the compactum $[\tfrac{1}{16},\alpha]$, we may choose $\varphi > 0$ such that for all $\tfrac{1}{16} \le \kappa \le \alpha$, $2 \kappa < (1 - \varphi) H(\kappa)$.
Moreover, we can put $\gamma := \min \{ (1-\varphi) H(\kappa) - 2 \kappa : 
\kappa \in \left[\tfrac{1}{16},\alpha\right]\}$ and have $\gamma > 0$.

By Lemma~\ref{asymptotic.lower.bound}, there is an $N_1$ such that $m+n \ge N_1$ yields
\[
\pr[\hoeffmss \geq 2^{(1-\varphi) \cdot H(\lambda) \cdot m}] \ge 1-\tfrac{\epsilon}{2},
\]
where $\lambda = \log_{1/q}(n) / m$.

By Lemma~\ref{squpperbound}, there is an $N_2$ such that 
$\pr[\lmsshalf > 2^{2 \log_{1/q}(n) + 1}] \leq \epsilon/2$ when $m+n\geq N_2$.
Let $N_3 = \max(N_1,N_2)$ and assume $m+n \ge N_3$.
With probability $1-\epsilon$,
\[
\hoeffmss / \lmsshalf
\geq 2^{(1-\varphi) \cdot H(\lambda) \cdot m} \cdot 2^{- 2 \log_{1/q}(n) - 1} \ge 2^{\gamma m - 1}.
\]
Thus, there is an $N \ge N_3$ such that $\hoeffmss / \lmsshalf \ge (1- 2\alpha)^{-1}$, whenever $m+n\geq N$.
Lemma~\ref{largeandsmall} (with $\delta = 0$ and $\nu = 1 - 2 \alpha$) completes the proof.
\end{proof}

\subsection{The case $q^{-\frac{m}{2}} \le n \le q^{-m^3}$}\label{sect.m.to.the.3}

Once the size $n$ of the right side reaches $q^{-\frac{m}{2}}$, 
the expected average left side of a maximal stable set becomes very close
to $\tfrac{m}{2}$, so close in fact that the methods developed so far
begin to fail: We cannot any longer show that the average is at most $\tfrac{m}{2}$.

The two main obstacles we face are: Firstly, when $n$ approaches $q^{-m/2}$
the upper bound on $\lmsshalf$, Lemma~\ref{squpperbound}, becomes useless as it reaches 
$2^m$. Secondly, the maximality requirement 
for maximal stable sets of small left side becomes harder to satisfy. Recall that we focus on 
left sides of size $\approx \log_{1/q}(n)$ because with this size the 
maximality requirement eliminates only a constant proportion of the possible small maximal stable sets. 
However, when $n$ surpasses $q^{-\frac{m}{2}}$, 
the maximal stable sets with left side $\log_{1/q}(n)$ can no longer be considered 
as \emph{small},
since $\log_{1/q}(n)>\tfrac{m}{2}$.

Therefore we lower our goals and aim instead for an average of at most $(\tfrac{1}{2}+\delta)m$,
for any given $\delta>0$.
Then 
the large maximal stable sets, those with left side $> (\tfrac{1}{2}+\delta)m$,
suddenly make up a significantly smaller proportion of the power set.

The key to that observation lies in the following basic lemma, 
a version of which can 
be found in, for instance, van Lint~\cite[Theorem~1.4.5]{LintBook99}.

\begin{lemma}\label{upper.bin}
For all $\tfrac{1}{2}<\gamma< 1$ 
it holds that
\[
\sum_{i=\lceil\gamma m\rceil}^{m}{m\choose i}\leq 2^{H(1-\gamma) m}.
\]
Moreover, $H(1-\gamma)<1$.
\end{lemma}
%

\begin{lemma}\label{giganticrightside}
For every $\delta>0$ and $\epsilon>0$ there is an $N$ and an $\alpha<\tfrac{1}{2}$ 
so that for $G\in\rbip(m,n;p)$
\[
\pr\left[\av\leq(\tfrac{1}{2}+\delta)m\right]\geq 1-\epsilon,
\]
for all $m,n$ with $m+n\geq N$ and $q^{-\alpha m} \le n \le q^{-m^3}$.
\end{lemma}

\begin{proof}
For $G\in\rbip(m,n;p)$, 
let us denote by $\lmssdelta$ the number of maximal stable sets $S$ with $|L(G) \cap S| \ge \left(\tfrac{1}{2} + \delta\right) m$.

Note that 
\begin{equation}\label{fewdeltalarge}
\lmssdelta \le \sum_{i=\lceil (1/2+\delta)m \rceil}^{m}{m\choose i}\leq
2^{H(1/2-\delta) m}, 
\end{equation}
where we used Lemma~\ref{upper.bin} for the second inequality.

We show now that, with high probability,  there are many more maximal stable sets
with small left side in a random graph $G\in\rbip(m,n;p)$, 
if $q^{-\alpha m} \le n \le q^{-m^3}$ and $m+n\geq N$, for an $N$ 
and an $\alpha<\tfrac{1}{2}$ 
that we will determine below. 
In order to do so, note first that 
we may assume  $\delta$ to be small enough so that 
$\alpha':=\tfrac{1}{2} - \tfrac{\delta}{3} \ge \tfrac{1}{16}$.
Next, fix $n'(m)=n' := \lceil q^{- \alpha' m} \rceil$ and 
delete arbitrary $n-n'$ vertices from $R(G)$. The resulting graph $G'$
may be viewed as a random graph in $\rbip(m,n';p)$, and we will see 
that with probability $\geq 1-\epsilon$ it contains many maximal stable sets
with small left side. More precisely, we will prove that 
\begin{equation}\label{deltaaim}
\mathcal{S}'_{G'}\geq \tfrac{3}{2\delta}\lmssdelta 
\end{equation}
with probability at least $1-\epsilon$. 
Note that the maximal stable sets counted by $\mathcal{S}'_{G'}$ 
have a left side of size  
\[
a'=\lfloor \log_{1/q}(\lfloor n' / m^{\log_{1/q}(m)} \rfloor) \rfloor\leq \alpha'm
= \left(1-\tfrac{2}{3}\delta\right)\tfrac{m}{2}.
\] 
Since every maximal stable set of $G'$ extends to a 
maximal stable set of $G$ with the same left side, we may then 
use Lemma~\ref{largeandsmall} with $\nu=\tfrac{2}{3}\delta$ to conclude
that $\av\leq\left(\tfrac{1}{2}+\delta\right)m$.

\medskip
Let us now see how we need to choose $N$ and $\alpha$ in order to guarantee~\eqref{deltaaim}, 
which is all we need to finish the proof.
For $\alpha$, we could take $\alpha'$ if it were not for the fact that we round 
up $q^{-\alpha'm}$ to get $n'$ (which turns out to be useful below). 
So we simply choose $\alpha$ to be somewhat larger than $\alpha'$:
Let $\alpha=\tfrac{1}{2} - \tfrac{\delta}{4}>\alpha'$ and choose $N_1$ 
large enough so that $q^{-\alpha m}\geq n'=\lceil q^{- \alpha' m} \rceil$
for all $m,n$ with $m+n\geq N_1$ and $n\leq q^{-m^3}$. 
(This is possible as $n\leq q^{-m^3}$ implies that $m$ has to be 
large as well if $m+n$ is large.)
Throughout the rest of the proof we will always assume that $m,n$ 
are integers with $q^{-\alpha m}\leq n\leq q^{-m^3}$.

Next, as $H$ is monotonously increasing in the 
interval $\left[0,\tfrac{1}{2}\right]$, it follows that 
\[
\varphi:=
1-\frac{H\left(\frac{1}{2} - \frac{\delta}{2}\right)}
{H\left(\alpha'\right)}
=1-\frac{H\left(\frac{1}{2} - \frac{\delta}{2}\right)}
{H\left(\frac{1}{2} - \frac{\delta}{3}\right)} < 1.
\] 
Applying Lemma~\ref{asymptotic.lower.bound} with $\epsilon$ and $\varphi$
yields an integer $N'$. Choose $N_2\geq N_1$ large enough so 
that $m+n\geq N_2$ implies $m+n'\geq N'$.
(Again, this is possible as $m$ has to be large if $m+n$ is large.)
Then, as 
$\alpha'  \ge \tfrac{1}{16}$, which in turn leads to $n' \ge q^{-\frac{m}{16}}$,
we obtain for $G'$ that
\[
\pr[\mathcal{S}'_{G'} \geq 2^{(1-\varphi) \cdot H(\log_{1/q}(n') / m) \cdot m}] \ge 1-\epsilon,
\]
for all $m$ with $m+n\geq N_2$.
Note that for $m,n$ with $m+n\geq N_2$ 
\[
H(\log_{1/q}(n') / m) = H(\log_{1/q}(\lceil q^{- \alpha' m} \rceil) / m) \ge H(\log_{1/q}(q^{- \alpha' m}) / m) = H(\alpha')
\]
and thus
\[
2^{(1-\varphi) \cdot H(\log_{1/q}(n') / m) \cdot m} \ge 2^{(1-\varphi) \cdot H(\alpha') \cdot m} = 2^{H(1/2 - \delta/2) \cdot m}.
\]
Put $\mu := H(1/2 - \delta/2)-H(1/2 - \delta)$ and
note that $\mu > 0$.
Inequality~\eqref{fewdeltalarge} gives that with probability $1-\epsilon$,
\begin{equation*}
\mathcal{S}'_{G'} / \lmssdelta \ge 2^{(H(1/2 - \delta/2)-H(1/2 - \delta)) \cdot m} = 2^{\mu m}.
\end{equation*}
Finally, choosing $N\geq N_2$ large enough so that $2^{\mu m}\geq \tfrac{3}{2}\delta$
for all $m,n$ with $m+n\geq N$ ensures~\eqref{deltaaim}.
Again, this is possible as $m$ grows with $m+n$.
\end{proof}

For the proof technique to work, we need $G'$ to be a large graph. Otherwise,
Lemma~\ref{asymptotic.lower.bound} cannot guarantee a high probability. 
In particular,  $m$ has to grow with $m+n$, which is why we assumed $n \le q^{-m^3}$.

\subsection{The case $q^{-m^3} \le n$ and 
proof of Theorem~\ref{realmainthm}}\label{sect.large.n}

If the right side of the random bipartite graph $G$ is huge in 
comparision to the left side, that is, if
$q^{-m^3} \le n$,
then almost surely $L(G)$ may be inductively matching into 
the right side. As a consequence, 
 the set of left sides of maximal stable sets 
is equal to the power set of $L(G)$, and thus $\av=\tfrac{m}{2}$.

\begin{lemma}\label{veryverylargeside}
For every $\epsilon>0$
there is an $N$ so that for $G\in\rbip(m,n;p)$ 
\[
\pr\left[\av\leq\tfrac{m}{2}\right]\geq 1-\epsilon,
\]
for all $m,n$ with $m+n\geq N$ and $q^{-m^3} \le n$.
\end{lemma}

\begin{proof}
Consider positive integers $m,n$ with $n \ge q^{-m^3}$.
We will give an $N$ such that for all such  $m,n$ with $m+n \ge N$ there are,
with probability $1 - \epsilon$, exactly $2^m$ maximal stable sets in $G$.
Then, every subset of $L(G)$ is the left side of a maximal stable set, 
which implies $\av = \tfrac{m}{2}$.

We proceed in a similar way as in
 the proof of Lemma~\ref{indmatchings} and therefore skip some of the details.
Let $R_1,\ldots,R_{\lfloor n/m\rfloor}$ be pairwise disjoint subsets of $R(G)$, of size $m$ each. 
For $i=1,\ldots,\lfloor n/m\rfloor$ let $M_i$
be the random indicator variable for an induced matching of size $m$ on $L(G)\cup R_i$.
Since $n \geq q^{-m^3}$, $n \geq m$ and so $\pr\left[M_i = 1 \right] \geq p^m q^{m^2}$.
Thus 
\[
\pr\left[\sum_{i=1}^{\lfloor n/m\rfloor}M_i=0\right]\leq \left(1-p^m q^{m^2} \right)^{\lfloor n/m\rfloor}
\leq e^{-p^m q^{m^2}{\lfloor n/m\rfloor}},
\]
where we use that $1+x \le e^x$ for all $x\in\mathbb R$.
Since $n \geq q^{-m^3}$, the term $p^m q^{m^2}{\lfloor n/m\rfloor}$ becomes arbitrarily large for large $m+n$. 
Thus, there is an $N$ so that 
$\pr\left[\sum_{i=1}^{\lfloor m/k\rfloor}M_i=0\right]\leq\epsilon$
for all $m,n$ with $m+n\geq N$ and $n \geq q^{-m^3}$.

Now assume that there is an induced matching on $L(G)\cup R_i$ of size $m$ for some $1 \le i \le \lfloor n/m\rfloor$.
Then are $2^m$ many maximal stable sets of the graph $G[L(G)\cup R_i]$ and each can be extended to a maximal stable set of $G$ without changing its left side. 
This completes the proof.
\end{proof}

Having exhausted all of the parameter space $(m,n)$, we may 
finally prove our main theorem.
\begin{proof}[Proof of Theorem~\ref{realmainthm}]
For given $\epsilon>0$ and $\delta>0$ choose $N_1$ and $\alpha<\tfrac{1}{2}$ as in
Lemma~\ref{giganticrightside}. Then let $N$ be at least as large as $N_1$ and the 
$N$ in Lemmas~\ref{largeleftside},~\ref{mdmsides},~\ref{largerightside},~\ref{hugerightside} (with $\alpha$ as chosen) and~\ref{veryverylargeside}. Then the theorem follows.
\end{proof}

\comment{
\section{Discussion}

A natural question arising is what happens when the edge-probability 
$p$ is no longer constant but depending on $N=m+n$. Although it seems that
our result could be generalized in that direction, we do not want to
complicate the analysis too much.


\begin{itemize}
\item Yes, we know, we can make $p$ dependent on $N$ -- however, we couldn't care 
less.
\item What does it mean for Frankl's conjecture?
\end{itemize}
}

\section*{Acknowledgements} We thank Carola Doerr for 
inspiring discussions and help with some of
the probabilistic arguments.
The second author is supported by a post-doc grant of the Fondation Sciences Math\'ematiques de Paris.

\bibliographystyle{amsplain}
\bibliography{../UnionClosedSurvey/survey/ucsbib}

\small
\vskip2mm plus 1fill
\noindent
Version 18 Feb 2013
\bigbreak

\noindent
Henning Bruhn
{\tt <bruhn@math.jussieu.fr>}\\
Oliver Schaudt
{\tt <schaudt@math.jussieu.fr>}\\[3pt]
Combinatoire et Optimisation\\
Universit\'e Pierre et Marie Curie\\
4 place Jussieu\\
75252 Paris cedex 05\\
France

\end{document}